\documentclass[a4paper, 11pt]{amsart}
\usepackage[latin1]{inputenc}
\usepackage[T1]{fontenc}
\usepackage[english]{babel}
\usepackage{amssymb}
\usepackage{amsmath}
\usepackage{amsthm}
\usepackage{amscd}
\usepackage{amsfonts}
\usepackage{stmaryrd}
\usepackage{pb-diagram}
\usepackage{epic,eepic,epsfig}
\usepackage{a4wide}
\usepackage{nextpage}
\usepackage{fancyhdr}

\newtheorem{prop}{Proposition}[section]
\newtheorem{corollary}[prop]{Corollary}

\newtheorem{lemma}[prop]{Lemma}

\newtheorem{theorem}[prop]{Theorem}

\theoremstyle{definition}
\newtheorem{remark}[prop]{Remark}
\newtheorem{definition}[prop]{Definition}
\newtheorem{notation}[prop]{Notation}

\begin{document}

\title{Modular invariants and isogenies}
\author{{F}abien {P}azuki}
\thanks{F. Pazuki, University of Copenhagen, fpazuki@math.ku.dk.}

\maketitle
\vspace{1cm}

\textsc{Abstract.}
We provide explicit bounds on the difference of heights of the $j$-invariants of isogenous elliptic curves defined over $\overline{\mathbb{Q}}$. The first one is reminiscent of a classical estimate for the Faltings height of isogenous abelian varieties, which is indeed used in the proof. We also use an explicit version of Silverman's inequality and isogeny estimates by Gaudron and R\'emond. We give applications in the study of V\'elu's formulas and of modular polynomials.

{\flushleft
 }

\begin{center}
---------
\end{center}

\begin{center}
\textbf{Invariants modulaires et isog\'enies}.
\end{center}

\textsc{R\'esum\'e.}
On donne dans ce texte des majorants pour la diff\'erence des hauteurs des $j$-invariants de courbes elliptiques isog\`enes d\'efinies sur $\overline{\mathbb{Q}}$. La premi\`ere in\'egalit\'e ressemble \`a un r\'esultat classique sur les hauteurs de Faltings de vari\'et\'es ab\'eliennes isog\`enes, qui s'av\`ere effectivement utile dans la preuve. On fera de plus usage d'une version explicite d'une in\'egalit\'e de Silverman et de travaux de Gaudron et R\'emond. On applique enfin ces r\'esultats aux formules de V\'elu et \`a l'\'etude des polyn\^omes modulaires.

{\flushleft
\textbf{Mots-Clefs:} Hauteurs, Courbes elliptiques, Isog\'enies.\\
}

\begin{center}
---------
\end{center}

\thispagestyle{empty}

\section{Introduction}

Two elliptic curves defined over a number field $K$ are isomorphic over $\overline{\mathbb{Q}}$ if and only if they have the same $j$-invariant. A natural question is: what happens to the $j$-invariant within an isogeny class? We obtain the following result in this direction.

\begin{theorem}\label{dim1}
 Let $\varphi:E_1\to E_2$ be a $\overline{\mathbb{Q}}$-isogeny between two elliptic curves defined over $\overline{\mathbb{Q}}$. Let $j_1$ and $j_2$ be the respective $j$-invariants. Then one has
\begin{equation}\label{6}
\vert h(j_{1})-h(j_{2})\vert \leq 9.204+12\log \deg\varphi,
\end{equation}
and
\begin{equation}\label{66}
 h(j_{1})-h(j_{2}) \leq 10.68+6\log \deg\varphi +6\log(1+h(j_1)),
\end{equation}
where $h(.)$ denotes the absolute logarithmic Weil height.
\end{theorem}

Inequality (\ref{6}) is more uniform, but inequality (\ref{66}) is better in several cases. For instance, to optimize inequalities (\ref{6}) and (\ref{66}), it is natural to consider an isogeny with minimal degree. In view of Th\'eor\`eme 1.4 of \cite{14}, recalled below in Theorem \ref{Gaudron Remond}, one obtains the following corollary.

\begin{corollary}\label{j isogenes}
Let $E_1$ and $E_2$ be semi-stable elliptic curves defined over a number field $K$ of degree $d$ and isogenous over $\overline{\mathbb{Q}}$, with respective $j$-invariants $j_1$ and $j_2$. Let $M=\max\{ h(j_1), h(j_2)\}$ and $m=\min\{h(j_{1}), h(j_{2})\}$. Then 
\begin{equation}
\vert h(j_{1})-h(j_{2}) \vert \leq 77.6+6\log(1+M)+12\log\Big(d\max\{m-14.16\,,\,11820\}+48d\log d \Big).
\end{equation}
Moreover, if $E_1$ (and hence $E_2$) has complex multiplications, then one has 
\begin{equation}
\vert h(j_{1})-h(j_{2}) \vert \leq 43.5+ 6\log(1+M)+12\log\Big(d\max\{m-14.16+ 6\log d, 12\} \Big).
\end{equation}
Finally, if $E_1$ and $E_2$ don't have complex multiplications and if $K$ has a real embedding, one has 
\begin{equation}
\vert h(j_{1})-h(j_{2}) \vert \leq 30+6\log(1+M)+ 12\log\Big( d\max\{m-14.16\,,\, 12\log d, 12\} \Big).
\end{equation}
\end{corollary}

\begin{remark}\label{elliptic isog}
There is no more degree of isogeny in the upper bounds of Corollary \ref{j isogenes}. Given two elliptic curves, it enables us to build a quick test to rule out the existence of a $\overline{\mathbb{Q}}$-isogeny between them. For instance let's take $j_{1}=2$ and $j_{2}=2^{474}$. These $j$-invariants are both integral, so the associated elliptic curves $E_1$ and $E_2$ have potentially good reduction everywhere, hence regarding basic reduction properties, nothing prevents these curves from being isogenous. Over the field $K=\mathbb{Q}(E_1[12])$, the curve $E_1$ is semi-stable (hence $E_2$ as well if the curves are isogenous), and $[K:\mathbb{Q}]\leq\# GL_2(\mathbb{Z}/12\mathbb{Z})=4608$. Now $\vert h(j_{1}) - h(j_{2})\vert= \log(2^{473})>77.6+6\log(1+\log2^{474})+12\log(4608\cdot11820+48\cdot4608\log4608)$, hence the curves are not isogenous over $\overline{\mathbb{Q}}$ by virtue of Corollary \ref{j isogenes}.
\end{remark}

\begin{remark}
Quantitative results on isogenous elliptic curves are first given in \cite{15, 16}, later improved in \cite{14}.
\end{remark}

\begin{remark}
Corollary \ref{j isogenes} implies finiteness of the set of $\overline{\mathbb{Q}}$-isomorphism classes of elliptic curves within an isogeny class. Indeed if $j_{1}$ is fixed and $K$ is fixed, then $j_{2}$ has bounded height, and the Northcott property of the Weil height concludes the argument. 
\end{remark}

We also obtain an explicit upper bound on the coefficients of modular polynomials. For any positive integer $m$, we denote by $\Phi_m$ the modular polynomial associated to cyclic isogenies of degree $m$. Let $\psi(m)=m\prod_{p\vert m}(1+p^{-1})$. We denote by $h_\infty(P)$ the logarithm of the maximum of the complex absolute values of the coefficients of a polynomial $P$. In Corollary \ref{modul}, we prove that for all $m\geq 1$
\begin{equation}
h_\infty(\Phi_m)\leq \psi(m)\Big(6\log m+\log\psi(m)+6\log(12\log m+2\log\psi(m)+25.2)+15.7  \Big),
\end{equation}
so the main term in the upper bound is slightly worse\footnote{We obtain $(7+\varepsilon)\psi(m)\log m$, the asymptotic is $(6-\varepsilon)\psi(m)\log m$.} than the known asymptotic, recalled in (\ref{Paula}), when $m$ grows to infinity. Estimates on modular polynomials are used for instance when computing explicitly Hilbert class polynomials, see for instance \cite{9, 5}.
\\

In order to prove Theorem \ref{dim1}, we compare the absolute logarithmic Weil height of the $j$-invariant with the Faltings height, then apply a classical estimate (Faltings \cite{12}, Raynaud \cite{17}) on the Faltings height in an isogeny class, in the particular case of elliptic curves here. Inequality (\ref{66}) comes from Silverman's work \cite{7} pages 254--258 (see (\ref{exsil}) below for explicit constants) and the isogeny estimate on the Faltings height. Previous comparisons between the $j$-invariant and the Faltings height were not sufficient, though, to get inequality (\ref{6}). In particular, estimates of the form $\vert h(j)-12h_F(E)\vert\leq c_1+c_2\log\max\{1,h\}$, with $c_1$ and $c_2$ positive constants and $h=h_F(E)$ or $h=h(j)$ are too weak. So the new input here is a modification of the Faltings height given in Notation \ref{modifiedFaltHeight} that encapsulates just enough of the complex elliptic curve data to make Lemma \ref{hauteur j} work. Once this lemma is established, a combination of the Faltings height estimate in an isogeny class and of Proposition \ref{alpha} leads to Theorem \ref{dim1}. The work of Gaudron and R\'emond \cite{14} is used multiple times. After introducing the material and proving Proposition \ref{alpha} in Section \ref{Definitions}, we prove Theorem \ref{dim1} and Corollary \ref{j isogenes} in Section \ref{j inv}. We apply Theorem \ref{dim1} to study the height of modular polynomials in Section \ref{modu}, and we add a remark on V\'elu's formulas in Section \ref{section velu}.

\section{Definitions and preliminaries}\label{Definitions}

\subsection{Basic notation}
\label{notbase}

If $K$ is a number field, we denote by $d$ its degree over $\mathbb{Q}$ and by $M_{K}$ the set of all places of $K$. For any natural prime number $p$, we normalize the archimedean absolute values by $\vert p\vert_v=p$ and the non-archimedean by $\vert p\vert_v=p^{-1}$ if $v$ divides $p$ or $\vert p\vert_v=1$ otherwise. For any $v\in{M_K}$, we have $K_{v}$ the completion of $K$ with respect to the valuation $|.|_{v}$.  We denote by $d_v$ the local degree $[K_v:\mathbb{Q}_v]$. The symbol $N_{K/\mathbb{Q}}$ denotes the norm on $K$ down to $\mathbb{Q}$. The notation $\log$ stands for the logarithm satisfying $\log(e)=1$. If $G$ is a finite set, the symbol $\#G$ stands for the number of elements in $G$.

If $\alpha$ is an algebraic number, element of a number field $K$, we will use the absolute logarithmic Weil height 
\begin{equation}
h(\alpha)=\frac{1}{d}\sum_{v\in{M_K}}d_v\log\max\{1,\vert\alpha\vert_v\}.
\end{equation}

\subsection{Discriminant and $j$-invariant}\label{discriminant}

Let $E$ be an elliptic curve defined over a number field $K$. Choose a Weierstrass equation $y^2=x^3+Ax+B$ where $A, B\in{K}$. Define $D=-16(4A^3+27B^2)$, which is non-zero because $E$ is a smooth curve. Define $j=-1728\frac{(4A)^3}{D}$. This element of $K$ is called the $j$-invariant of $E$. Two elliptic curves defined over $K$ are isomorphic over $\overline{\mathbb{Q}}$ if and only if they have the same $j$-invariant.

Choose an embedding of $K$ in $\mathbb{C}$. Let $\mathbb{H}$ be the complex upper half plane. We associate to $E$ a complex number in the fundamental domain $\mathcal{F}=\{\tau\in{\mathbb{H}\,\vert\, \vert \tau\vert\geq 1 \; \mathrm{and}\; \vert \mathrm{Re}(\tau)\vert\leq \frac{1}{2}}\}$. It satisfies in particular $\mathrm{Im} \tau\geq \sqrt{3}/2$. As explained in Proposition 1.5 page 10 of \cite{18}, the choice is not unique. Among the choices we have we agree on taking, for instance, the $\tau$ with smallest real part. We will call this unique $\tau$ the \emph{reduced} $\tau$. 

The $j$-invariant of $E$ then becomes a modular function $\mathbb{H}\to \mathbb{C}$. We will also use the modular discriminant $\Delta$. Let us introduce the notation $q=e^{2\pi i\tau}$, then $j$ and $\Delta$ become naturally functions of $q$ by looking at their Fourier expansion at infinity. Several choices of normalization appear in the literature. We favor $\Delta(q)=q\prod_{n=1}^{+\infty}(1-q^n)^{24}$, without the multiplicative factor $(2\pi)^{12}$, and we took care of keeping all constants coherent with this choice. The beginning of the $q$-expansion of the $j$-invariant is $\frac{1}{q}+744+196884q+\ldots$.

\subsection{The Faltings height}

 Let $E$ be a semi-stable elliptic curve defined over a number field $K$. Let $S=\mathrm{Spec}({\mathcal O}_K)$, where
${\mathcal O}_K$ is the ring of integers of $K$ and let $\pi\colon {\mathcal E}\longrightarrow S $
be the N\'eron model of $E$ over $S$. Denote by $\varepsilon\colon S\longrightarrow {\mathcal E}$ the zero section of $\pi$ and by $\omega_{{\mathcal E}/S}$ the sheaf of relative differentials
\begin{equation}
\omega_{{\mathcal E}/S}:=\varepsilon^{\star}\Omega_{{\mathcal
E}/S}\simeq\pi_{\star}\Omega_{{\mathcal E}/S}\;.
\end{equation}

For any archimedean place $v$ of $K$, denote by $\sigma$ an embedding of $K$ in $\mathbb{C}$ associated to $v$, then the corresponding line bundle
\begin{equation}
\omega_{{\mathcal E}/S,\sigma}=\omega_{{\mathcal E}/S}\otimes_{{\mathcal O}_K,\sigma}\mathbb{C}\simeq H^0({\mathcal
E}_{\sigma}(\mathbb{C}),\Omega_{{\mathcal E}_\sigma}(\mathbb{C}))\;
\end{equation}
can be equipped with a natural $L^2$-metric $\Vert.\Vert_{\sigma}$ defined by
\begin{equation}
\Vert\alpha\Vert_{\sigma}^2=\frac{i}{2}\int_{{\mathcal
E}_{\sigma}(\mathbb{C})}\alpha\wedge\overline{\alpha}\;.
\end{equation}

The ${\mathcal O}_K$-module of rank one $\omega_{{\mathcal E}/S}$, together with the hermitian norms
$\Vert.\Vert_{\sigma}$ at infinity defines an hermitian line bundle 
$\overline{\omega}_{{\mathcal E}/S}=({\omega}_{{\mathcal E}/S}, (\Vert .\Vert_\sigma)_{\sigma:K\hookrightarrow \mathbb{C}})$ over $S$, which has a well defined Arakelov degree
$\widehat{\mathrm{deg}}(\overline{\omega}_{{\mathcal E}/S})$. Recall that for any hermitian line
bundle $\overline{\mathcal L}$ over $S$, the Arakelov degree of $\overline{\mathcal L}$
is defined as
\begin{equation}
\widehat{\mathrm{deg}}(\overline{\mathcal L})=\log\#\left({\mathcal L}/{s{\mathcal
O}}_K\right)-\sum_{\sigma:K\hookrightarrow \mathbb{C}}\log\Vert
s\Vert_{\sigma}\;,
\end{equation}
where $s$ is any non zero section of $\mathcal L$, and the result does not depend on the choice of $s$ in view of the product formula in $K$. 

We now give the definition of the classical Faltings height of \cite{12}.

\begin{definition}\label{faltings}  The Faltings height of $E$ is defined as
\begin{equation}
h_F(E):=\frac{1}{[K:\mathbb{Q}]}\widehat{\mathrm{deg}}(\overline{\omega}_{{\mathcal
E}/S}).
\end{equation}
\end{definition}

A key property is given by Raynaud in Corollaire 2.1.4 point (1) page 207 of \cite{17} (it can also be deduced from Lemma 5 page 358 of \cite{12}), any two elliptic curves $E_1$ and $E_2$ with an isogeny $\varphi:E_1\to E_2$ satisfy
\begin{equation}\label{Raynaud}
\vert h_F(E_1)- h_F(E_2)\vert\leq\frac{1}{2}\log\deg\varphi.
\end{equation}

\subsection{Injectivity diameter}

We recall the definition of the injectivity diameter as it will be useful in Lemma \ref{rho}.
Let $A$ be a complex abelian variety and $L$ be a polarization on $A$. It induces a hermitian norm $\Vert.\Vert_L$ on the tangent space $t_A$ by setting $\Vert z\Vert_L=\sqrt{H(z,z)}$ for $z\in{t_A}$, where $H(.,.)$ is the Riemann form associated to $L$. Let $\Omega_A$ be the period lattice of $A$. Then we define 
\begin{equation}
\rho(A,L)=\min\{\Vert\omega\Vert_L \;\vert\; \omega\in{\Omega_A},\, \omega\neq0\}.
\end{equation}
The number $\rho(A,L)$ acquires the name \emph{injectivity diameter} as it is the diameter of the largest ball on which the exponential $\exp: t_A\to A$ is injective.

In dimension 1, if one normalizes $\Omega_A=\mathbb{Z}+\tau\mathbb{Z}$ then $H(z,w)=\frac{z\overline{w}}{\mathrm{Im}\tau}$, see Example 4.1.3 page 71 of \cite{3}.

\subsection{Upper half plane and isogenies}

Let $E$ be an elliptic curve over $\mathbb{C}$. In the sequel, we always use the reduced $\tau$ introduced earlier and such that $E(\mathbb{C})\simeq\mathbb{C}/(\mathbb{Z}+\tau\mathbb{Z})$.

\begin{notation}
Let $E_1$ and $E_2$ be elliptic curves defined over a number field $K$. Let $\sigma: K\hookrightarrow\mathbb{C}$ be a complex embedding and $ E_{1}(\mathbb{C})\simeq\mathbb{C}/(\mathbb{Z}+\tau_{1,\sigma}\mathbb{Z})$, respectively $E_{2}(\mathbb{C})\simeq\mathbb{C}/(\mathbb{Z}+\tau_{2,\sigma}\mathbb{Z})$, where we ask that $\tau_{1,\sigma}$ and $\tau_{2,\sigma}$ are reduced.
We will use 
\begin{equation}
\alpha(E_1,E_2)=\frac{1}{2[K:\mathbb{Q}]}\sum_{\sigma:K\hookrightarrow \mathbb{C}}\log\frac{ \mathrm{Im}\tau_{1,\sigma}}{\mathrm{Im}\tau_{2,\sigma}}.
\end{equation}
\end{notation}

Our goal in this paragraph is to bound from above $\alpha(E_1,E_2)$ in the case where $E_1$ and $E_2$ are isogenous elliptic curves. We will use the following general lemma.

\begin{lemma}\label{rho}
Let $\varphi:A_1\to A_2$ be an isogeny and $L$ a polarization on $A_2$. Then 
\begin{equation}
\rho(A_2,L)\leq \rho(A_1,\varphi^{*}L)\leq (\deg\varphi)\rho(A_2,L),
\end{equation}
where $\rho(A,M)$ denotes the injectivity diameter of the polarized abelian variety $(A,M)$.
\end{lemma}
\begin{proof}
This is Lemma 3.4 page 356 of \cite{14}.
\end{proof}

\begin{lemma}\label{alphaC}
Let $E_1$ and $E_2$ be two elliptic curves over $\mathbb{C}$. Let us denote $E_{1}(\mathbb{C})\simeq\mathbb{C}/(\mathbb{Z}+\tau_{1}\mathbb{Z})$ and $E_{2}(\mathbb{C})\simeq\mathbb{C}/(\mathbb{Z}+\tau_{2}\mathbb{Z})$, where we ask that $\tau_{1}$ and $\tau_{2}$ are reduced. Suppose there is an isogeny $\varphi: E_1\to E_2$. Then
\begin{equation}
-\log\deg\varphi\leq\log\frac{\mathrm{Im}\tau_1}{\mathrm{Im}\tau_2}\leq \log\deg\varphi.
\end{equation}
\end{lemma}

\begin{proof}
Any polarization on $E_1$ is a power of the principal polarization $L_1$ on $E_1$. Let us take the principal polarization $L_2$ on $E_2$, so one has $\varphi^{*}L_2=L_1^{\otimes \deg\varphi}$. Thus $\rho(E_1,\varphi^{*}L_2)^2=(\deg\varphi)\rho(E_1,L_1)^2$. Apply Lemma \ref{rho} (in dimension 1) to get
\begin{equation}\label{ineq1}
\rho(E_2,L_2)\leq (\deg\varphi)^{\frac{1}{2}}\rho(E_1,L_1)\leq (\deg\varphi)\rho(E_2,L_2).
\end{equation}
Now remark that $\rho(E_1,L_1)^{-2}=\mathrm{Im}\tau_1$ and $\rho(E_2,L_2)^{-2}=\mathrm{Im}\tau_2$ (or just refer to Remarque 3.3 page 356 of \cite{14}), hence (\ref{ineq1}) becomes 
\begin{equation}\label{ineq2}
\mathrm{Im}\tau_2\geq (\deg\varphi)^{-1}\mathrm{Im}\tau_1\geq (\deg\varphi)^{-2}\mathrm{Im}\tau_2,
\end{equation}
which gives 
\begin{equation}
(\deg\varphi)^{-1}\leq\frac{\mathrm{Im}\tau_1}{\mathrm{Im}\tau_2}\leq\deg\varphi,
\end{equation}
this concludes the proof.
\end{proof}

We give now an analytic lemma about the $j$-invariant, which improves on Lemme 1 page 187 of \cite{11} and on (3) page 2.6 of \cite{10}. 

\begin{lemma}\label{chemin faisant}
Let $\tau$ be an element of the upper half plane. Then 
\begin{equation}
\vert j(\tau)\vert\geq e^{2\pi\mathrm{Im}\tau}-970.8.
\end{equation}
\end{lemma}

\begin{proof}
Let us start by writing the $q$-expansion of $j$ as $\displaystyle{j(\tau)=\frac{1}{q}+\sum_{n=0}^{+\infty}c_n q^n}$. All the coefficients $c_n$ are positive integers. Indeed, if one denotes $\displaystyle{\sigma_{3}(n)=\sum_{d\vert n}d^3}$, one has the classical formula (see for instance \cite{18} Proposition 7.4.b page 60) 
\begin{equation}\label{jdeltag}
j(\tau)=\frac{\displaystyle{\Big(1+240\sum_{n=1}^{+\infty}\sigma_{3}(n)q^n}\Big)^3}{\displaystyle{q\prod_{n=1}^{+\infty}(1-q^n)^{24}}},
\end{equation}
and each factor $\displaystyle{(1-q^n)^{-24}=(\sum_{m=0}^{+\infty}q^{mn})^{24}}$ has positive integral coefficients, hence $j$ as well.

Let us now write $\tau=x+iy$ with $x,y\in{\mathbb{R}}$ and $y>0$. Then $j(\tau)=e^{-2\pi ix} e^{2\pi y}+f(x,y)$ where $\displaystyle{f(x,y)=\sum_{n=0}^{+\infty} c_n e^{2\pi i nx} e^{-2\pi ny}}$. One has $\displaystyle{\vert f(x,y)\vert \leq \sum_{n=0}^{+\infty} c_n e^{-2\pi ny}=j(iy)-e^{2\pi y}=g(y)}$, where the function $g$ is positive and decreasing on $]0,+\infty[$, because all the coefficients $c_n$ are positive. 

Let $y_0$ be the unique positive real number such that $j(iy_0)=2e^{2\pi y_0}$. Then for all $y\geq y_0$, one has $g(y)\leq g(y_0)=e^{2\pi y_0}$, so we get $\vert j(\tau)\vert\geq e^{2\pi y}-\vert f(x,y)\vert\geq e^{2\pi y}-e^{2\pi y_0}$ for any $y\geq y_0$, and for $y<y_0$ the inequality $\vert j(\tau)\vert\geq e^{2\pi y}-e^{2\pi y_0}$ holds trivially. We now need to estimate $y_0$ from above.

We will use an inversion process \textit{\`a la} Ramanujan. By (2.1) page 430 and (2.8) page 431 of \cite{2} one has the following formulas, where $0<\alpha<1$ 
\begin{equation}
j=27\frac{(1+8\alpha)^3}{\alpha(1-\alpha)^3}\quad\quad \mathrm{and}\quad q=e^{2\pi i\tau}=\exp\Big(-\frac{2\pi}{\sqrt{3}} \frac{_2F_1(\frac{1}{3},\frac{2}{3},1,1-\alpha)}{_2F_1(\frac{1}{3},\frac{2}{3},1,\alpha)} \Big),
\end{equation}
where $_2F_1$ is a classical hypergeometric function. There exists a real $\alpha_0$ such that 
\begin{equation}
y_0=\frac{1}{\sqrt{3}} \frac{_2F_1(\frac{1}{3},\frac{2}{3},1,1-\alpha_0)}{_2F_1(\frac{1}{3},\frac{2}{3},1,\alpha_0)},
\end{equation}
and it must satisfy, by definition of $y_0$, 
\begin{equation}
2\exp\Big(2\pi \frac{1}{\sqrt{3}} \frac{_2F_1(\frac{1}{3},\frac{2}{3},1,1-\alpha_0)}{_2F_1(\frac{1}{3},\frac{2}{3},1,\alpha_0)}\Big)=27\frac{(1+8\alpha_0)^3}{\alpha_0(1-\alpha_0)^3}.
\end{equation}
A quick numerical interval search provides us with $\alpha_0$ being close to $\alpha_1=0.02739$, and as 
\begin{equation}
2\exp\Big(2\pi \frac{1}{\sqrt{3}} \frac{_2F_1(\frac{1}{3},\frac{2}{3},1,1-\alpha_1)}{_2F_1(\frac{1}{3},\frac{2}{3},1,\alpha_1)}\Big)>27\frac{(1+8\alpha_1)^3}{\alpha_1(1-\alpha_1)^3},
\end{equation}
we obtain $e^{2\pi y_0}\leq \exp\Big(2\pi \frac{1}{\sqrt{3}} \frac{_2F_1(\frac{1}{3},\frac{2}{3},1,1-\alpha_1)}{_2F_1(\frac{1}{3},\frac{2}{3},1,\alpha_1)}\Big)  \leq 970.8$.
\end{proof}

We are now ready to prove the following lemma, which will be useful in the proof of inequality (\ref{66}).

\begin{lemma}\label{logtau}
Let $E$ be an elliptic curve defined over a number field $K$. Let $\sigma:K\hookrightarrow \mathbb{C}$ be a complex embedding and let $\tau_\sigma$ be the reduced element of the upper half plane corresponding to $E$. Then 
\begin{equation}
\frac{1}{2}\log\frac{\sqrt{3}}{2}\leq \frac{1}{2[K:\mathbb{Q}]}\sum_{\sigma:K\hookrightarrow \mathbb{C}} \log\mathrm{Im}\tau_\sigma\leq \frac{1}{2}\log(1+h(j))+0.97-\frac{1}{2}\log 2\pi.
\end{equation}
\end{lemma}

\begin{proof}
For any complex embedding $\sigma$, as $\tau_\sigma$ is reduced one has $\mathrm{Im}\tau_\sigma\geq \sqrt{3}/2$. Together with Lemma \ref{chemin faisant}, we have 
\begin{equation}
\frac{\sqrt{3}}{2}\leq \mathrm{Im}\tau_\sigma\leq\frac{1}{2\pi}\log(\vert j\vert_\sigma +970.8),
\end{equation}
hence for $d=[K:\mathbb{Q}]$, 
\begin{equation}
\log\frac{\sqrt{3}}{2}\leq \frac{1}{d}\sum_{\sigma:K\hookrightarrow \mathbb{C}}\log\mathrm{Im}\tau_\sigma\leq  \frac{1}{d}\sum_{\sigma:K\hookrightarrow \mathbb{C}}\log\Big(\frac{1}{2\pi}\log(\vert j\vert_\sigma +970.8)\Big),
\end{equation}
and a direct estimate gives $\log\Big(\frac{1}{2\pi}\log(\vert j\vert_\sigma +970.8)\Big)\leq \log\log\max\{\vert j\vert_\sigma, e\}+1.94-\log2\pi$, one concludes by arithmetico-geometric mean on the sum over all embeddings, as in (11) page 258 of \cite{7}.

\end{proof}

\begin{prop}\label{alpha}
Let $E_1$ and $E_2$ be two elliptic curves over a number field $K$, isogenous over $\overline{\mathbb{Q}}$. Then 
\begin{equation}
\alpha(E_1,E_2)\leq  \frac{1}{2}\log\deg\varphi.
\end{equation}
Moreover, one also has 
\begin{equation}
\alpha(E_1,E_2)\leq \frac{1}{2}\log(1+h(j_1))+0.97-\frac{1}{2}\log 2\pi -\frac{1}{2}\log\frac{\sqrt{3}}{2}.
\end{equation}
\end{prop}
\begin{proof}
Straightforward from Lemma \ref{alphaC} and Lemma \ref{logtau}.
\end{proof}

\section{Proof of Theorem \ref{dim1}}\label{j inv}

Let us introduce the following quantity.

\begin{notation}\label{modifiedFaltHeight} 
Let $E$ be an elliptic curve over a number field $K$. For each complex embedding $\sigma: K\hookrightarrow\mathbb{C}$, choose the reduced $\tau_\sigma$ such that $E_{\sigma}(\mathbb{C})\simeq\mathbb{C}/(\mathbb{Z}+\tau_{\sigma}\mathbb{Z})$. We will use the following number,
\begin{equation}
h_\nu(E)=h_F(E)+\frac{1}{2[K:\mathbb{Q}]}\sum_{\sigma:K\hookrightarrow \mathbb{C}}\log \mathrm{Im}\tau_{\sigma} + \log2\pi.
\end{equation}
\end{notation}

We will prove Theorem \ref{dim1} by using a comparison between the $j$-invariant and $h_{\nu}(E)$. Here is the key lemma. The first inequality is a modified version of Lemma 7.9 page 393 of \cite{14}.

\begin{lemma}\label{hauteur j} For any elliptic curve $E$ over $\overline{\mathbb{Q}}$ with j-invariant
$j$ one has 
\begin{equation}
-0.583\le{1\over12}h(j)-h_{\nu}(E)\leq 0.184.
\end{equation}
\end{lemma}

\begin{proof} Throughout the proof, we will always consider reduced elements $\tau_\sigma$. We start by applying formula (10) of Silverman~\cite{7} page 257, if $E$ is defined and semi-stable over a number field $K$, denote the minimal discriminant of $E$ by $D_{E/K}$ and the degree of $K$ by $d=[K:\mathbb{Q}]$, then
\begin{equation}\label{Silverman}
h(j)={1\over d}\log|N_{K/\mathbb{Q}}D_{E/K}|+{1\over
d}\sum_{\sigma:K\hookrightarrow \mathbb{C}}
\log\max\{1,|j(\tau_\sigma)|\}.
\end{equation}

We will use Proposition 1.1 of
\cite{7} page 254, where one has to correct a power of $2\pi$ in the definition of $\Delta$ (see paragraph \ref{discriminant}) for the formula to hold, as already done in paragraph 3 page 426 of \cite{1} or in Proposition 8.2 page 195 of \cite{8} (his $\Delta$ is given in Definition 4.4 page 185). See also Theorem 7 page 419 of \cite{13} (his $\Delta$ is defined page 416). We obtain the following.

\begin{equation}\label{SilJong}
h_F(E)={1\over12d}
\log|N_{K/\mathbb{Q}}D_{E/K}|-{1\over12d}\sum_{\sigma:K\hookrightarrow \mathbb{C}}\log((2\pi)^{12}|\Delta(\tau_\sigma)|(\mathrm{Im}\tau_\sigma)^6),
\end{equation}
thus
\begin{equation}\label{hV}
h_{\nu}(E)={1\over12d}
\log|N_{K/\mathbb{Q}}D_{E/K}|-{1\over12d}\sum_{\sigma:K\hookrightarrow\mathbb{C}}\log|\Delta(\tau_\sigma)|.
\end{equation}

Thus by substracting $(\ref{Silverman})$ to $(\ref{hV})$ we get the key equality
\begin{equation}\label{diff}
h_{\nu}(E)-{1\over12}h(j)=-{1\over12d}
\sum_{\sigma:K\hookrightarrow\mathbb{C}}{d_v\log\max\{|\Delta(\tau_\sigma)|,\vert
j(\tau_\sigma) \Delta(\tau_\sigma)\vert\}}.
\end{equation}

We now prove the first inequality of the lemma. We have
\begin{equation}
-\log|\Delta(\tau_{\sigma})|\leq
2\pi\mathrm{Im}\tau_{\sigma}+24C_{\tau_{\sigma}}, 
\end{equation}
with
\begin{equation}
C_{\tau_{\sigma}}= -\sum_{n=1}^{+\infty}\log|1-q^{n}|\leq
-\sum_{n=1}^{+\infty}\log(1-e^{-2\pi\mathrm{Im}\tau_{\sigma} n})\leq -\sum_{n=1}^{+\infty}\log(1-e^{-\pi\sqrt{3} n}), 
\end{equation}
where we used the inequality $\mathrm{Im}\tau_\sigma\geq \sqrt{3}/2$. A direct estimate then gives $24C_{\tau_\sigma}\leq 1/9$. 
That provides us with 
\begin{equation}
|\Delta(\tau_\sigma)|\ge e^{-1/9-2\pi \mathrm{Im}\tau_\sigma}.
\end{equation}
One has
$|j(\tau_\sigma)|\ge e^{2\pi \mathrm{Im}\tau_\sigma}-970.8$ by Lemma \ref{chemin faisant}. Hence we get \begin{equation}
\max\{1,\vert j(\tau_\sigma)\vert\}
|\Delta(\tau_\sigma)|\ge\max\{1,e^{2\pi \mathrm{Im}\tau_\sigma}-970.8\}
 e^{-1/9-2\pi \mathrm{Im}\tau_\sigma},
\end{equation}
which can be written
\begin{equation}
\max\{1,\vert j(\tau_\sigma)\vert\}
|\Delta(\tau_\sigma)| \ge e^{-1/9}F(\mathrm{Im}\tau_\sigma)
\end{equation}
where $F$ is the function
given by 
\begin{equation}
F(y)=\max\{e^{-2\pi y}, 1-970.8 e^{-2\pi y}\}.
\end{equation}
The inequality $1-970.8 e^{-2\pi y}\geq e^{-2\pi y}$ is equivalent to $y\geq \log(971.8)/2\pi$. On the interval $[\sqrt{3}/2, \log(971.8)/2\pi]$, the function $F$ is decreasing and $F(y)\geq F(\log(971.8)/2\pi)$ and on $[\log(971.8)/2\pi, +\infty[$, the function $F$ is increasing and $F(y)\geq F(\log(971.8)/2\pi)$ as well, hence $\forall y\geq \frac{\sqrt{3}}{2},\, F(y)\geq F(\log(971.8)/2\pi)=1/971.8$. It allows us to conclude by injecting in (\ref{diff})
 \begin{equation}\label{first}
 h_{\nu}(E)-{1\over12}h(j)\leq -\frac{1}{12}\log\Big(\frac{e^{-1/9}}{971.8}\Big)\leq 0.583.
 \end{equation}

Let us prove the second inequality. We have (see for instance \cite{11} page 184) the following classical equality equivalent to (\ref{jdeltag}), for any $\tau_\sigma$ in the upper half plane, 
\begin{equation}
j(\tau_\sigma)\Delta(\tau_\sigma)=\Big(1+240\sum_{n=1}^{+\infty}n^3\frac{q^n}{1-q^n}\Big)^3.
\end{equation}
Using $\mathrm{Im}\tau_\sigma\geq\sqrt{3}/2$ for the reduced $\tau_{\sigma}$ and $\vert 1-q^n\vert\geq 1-\vert q\vert^n$ one has by direct estimate
\begin{equation}\label{jdelta1}
\vert j(\tau_\sigma) \Delta(\tau_\sigma)\vert\leq \Big(1+240\Big\vert\sum_{n=1}^{+\infty}n^3\frac{q^n}{1-q^n}\Big\vert\Big)^3\leq \Big(1+240\sum_{n=1}^{+\infty}n^3\frac{e^{-\pi \sqrt{3} n}}{1-e^{-\pi\sqrt{3} n}}\Big)^3\leq 9.02.
\end{equation}  We also have 
\begin{equation}
\log\vert\Delta(\tau_\sigma)\vert = \log\vert q\vert+24\sum_{n=1}^{+\infty}\log\vert1-q^n\vert,
\end{equation}
hence using again $\mathrm{Im}\tau_\sigma\geq\sqrt{3}/2$ we obtain 
\begin{equation}\label{jdelta2}
\log\vert\Delta(\tau_\sigma)\vert\leq -\pi\sqrt{3}+24\sum_{n=1}^{+\infty}\log(1+e^{-\pi\sqrt{3}n})\leq -5,
\end{equation} so we can bound from above the maximum
\begin{equation}
\log\Big(\max\{\vert\Delta(\tau_\sigma)\vert,\vert j(\tau_\sigma)\Delta(\tau_\sigma)\vert\}
\Big)\leq \log(9.02),
\end{equation}
by comparing $(\ref{jdelta1})$ and $(\ref{jdelta2})$.
Inject in (\ref{diff}) to get the second inequality of the lemma, which together with (\ref{first}) concludes the whole proof. 
\end{proof}

\textit{Final step in the proof of Theorem \ref{dim1}.}
Let $E_1$ and $E_2$ be two elliptic curves over $\overline{\mathbb{Q}}$. Suppose $\varphi:E_1\to E_2$ is an isogeny. Let $j_1$ and $j_{2}$ be the corresponding $j$-invariants. We start by inequality (\ref{6}).
One writes:
\\

\begin{tabular}{lll}
$\!\!\!\!\!\!\!\!\!\frac{1}{12}\left( h(j_1)-h(j_2)\right)$ & $=$ & $\frac{1}{12}h(j_1)-h_{\nu}(E_1)+ h_{\nu}(E_1)-h_{\nu}(E_2)+ h_{\nu}(E_2)-\frac{1}{12}h(j_2)$\\
\\
$$ & $\leq$ & $0.184+  h_{\nu}(E_1)-h_{\nu}(E_2)   +0.583 \hfill $\\
\\
$$ & $\leq$ & $0.767+  h_{F}(E_1)-h_{F}(E_2)  +  \alpha(E_1,E_2) \hfill $\\
\\
$$ & $\leq$ & $0.767+ \frac{1}{2}\log \deg\varphi+ \frac{1}{2}\log\deg\varphi,\hfill $\\
\\
\end{tabular}

\noindent using the two inequalities of Lemma \ref{hauteur j}, then the isogeny estimate (\ref{Raynaud}) and the first inequality of Proposition \ref{alpha}. Finally, use the dual isogeny to obtain the same upper bound for the opposite difference and conclude.

For inequality (\ref{66}), the calculation starts in the same way but the estimate on $\alpha(E_1, E_2)$ is given using the second inequality of Proposition \ref{alpha}.
\\

\begin{tabular}{lll}
$\frac{1}{12}\left( h(j_1)-h(j_2)\right)$ & $\leq$ & $0.767+  h_{F}(E_1)-h_{F}(E_2)  +  \alpha(E_1,E_2) \hfill $\\
\\

$$ & $\leq$ & $0.767+ \frac{1}{2}\log \deg\varphi+ \frac{1}{2}\log(1+h(j_1))\hfill $\\
\\

$$ & $$ & $+0.97-\frac{1}{2}\log 2\pi -\frac{1}{2}\log\frac{\sqrt{3}}{2}\hfill$\\
\\

$$ & $\leq$ & $0.89+\frac{1}{2}\log \deg\varphi+ \frac{1}{2}\log(1+h(j_1)).$\\
\\
\end{tabular}

This concludes the proof of Theorem \ref{dim1}. Note that Lemma \ref{hauteur j} and Lemma \ref{logtau} give the following explicit version of Silverman's comparison:

\begin{equation}\label{exsil}
1.18\leq \frac{1}{12}h(j)- h_{F}(E)\leq 2.08 + \frac{1}{2}\log(1+h(j)).
\end{equation}
\\

To get Corollary \ref{j isogenes}, let us now recall Th\'eor\`eme 1.4 page 347 of \cite{14}.
\begin{theorem}(Gaudron-R\'emond)\label{Gaudron Remond}
Let $E$ and $E'$ be elliptic curves defined over a number field $K$ of degree $d$, isogenous over $\overline{K}$. Then there exists a $\overline{K}$-isogeny $\varphi:E\to E'$ such that 
\begin{equation}
\deg\varphi \leq 10^{7}d^2 (\max\{h_{F}(E), 985\}+4\log d )^2.
\end{equation}
Moreover, if $E$ (and hence $E'$) has complex multiplications, one has the better upper bound 
\begin{equation}
\deg\varphi \leq 34000d^2 (\max\{h_{F}(E)+ \frac{1}{2}\log d, 1\})^2.
\end{equation}
Finally, if $E$ and $E'$ don't have complex multiplications and if $K$ has a real embedding, then one has the better upper bound 
\begin{equation}
\deg\varphi \leq 3583 d^2 (\max\{h_{F}(E), \log d, 1\})^2.
\end{equation}
\end{theorem}

One may well replace $h_{F}(E)$ by $\min\{h_F(E), h_F(E')\}$ in the three inequalities of Theorem \ref{Gaudron Remond}, by considering the dual isogeny. Apply successively Theorem \ref{dim1}, Theorem \ref{Gaudron Remond} and (\ref{exsil}) to deduce Corollary \ref{j isogenes} with $10.68+6\log(10^7/144)\leq 77.6$ for the first case, $10.68+6\log(34000/144)\leq 43.5$ for the second case and $10.68+6\log(3583/144)\leq 30$ for the last case.

\section{Modular polynomials}\label{modu}

For a positive integer $m$, the modular polynomial $\Phi_m$ is the minimal polynomial of $j(mz)$ over the field $\mathbb{C}(j(z))$. It is a polynomial in two variables $\Phi_m(X,Y)\in{\mathbb{Z}[X,Y]}$, satisfying $\Phi_m(X,Y)=\Phi_m(Y,X)$ and $\Phi_m(j(mz),j(z))=0$. Its degree in each variable is $\psi(m)=m\prod_{p\vert m}(1+p^{-1})$.
Let $j_0\in{\overline{\mathbb{Q}}}$ be fixed, corresponding to the elliptic curve $E_0$. Then the roots of $\Phi_m(X,j_0)$ are exactly the $j$-invariants of elliptic curves with a cyclic isogeny of degree $m$ to $E_0$.

The coefficients of $\Phi_m$ grow rather rapidly with $m$. We recall the asymptotic result given in \cite{6}. Denote the height of a polynomial in $\mathbb{\mathbb{C}}[X,Y]$ by 
\begin{equation}
h_\infty\Big(\sum_{0\leq s,k\leq n}c_{s,k}X^sY^k\Big)=\log\max_{0\leq s,k\leq n}\vert c_{s,k}\vert.
\end{equation}

When $m$ goes to infinity \cite{6} provides us with 
\begin{equation}\label{Paula}
h_\infty(\Phi_m)=6\psi(m)\Big(\log m -\sum_{p\vert m}\frac{1}{p}\log p +O(1)\Big).
\end{equation}  

In the case where $m=\ell$ is a prime number, one finds in \cite{4} the explicit inequality 
\begin{equation}
h_\infty(\Phi_\ell)\leq 6\ell \log\ell +16\ell +14\sqrt{\ell}\log\ell.
\end{equation}

We will give an explicit upper bound valid for general $m$, but slightly worse than the previous bound for the prime case. We start by a lemma, essentially Lemma 20 page 312 of \cite{4}, with a statement that will better work for us.

\begin{lemma}\label{lem20}
Let $P\in{\mathbb{C}[X,Y]}$ be a nonzero polynomial of degree at most $n\geq1$ in each variable. Suppose $h_\infty(P(X,y_k))\leq B$ for each $y_k=L (1+\frac{k}{n})$, with $0\leq k \leq n$, for some real numbers $B>0$ and $L>1$. Then we have 
\begin{equation}
h_\infty(P)\leq B+\Big(\frac{1+ \log L}{L}+3\log2\Big)n.
\end{equation}
\end{lemma}

\begin{proof}
We may write $P(X,Y)=\sum_{0\leq r\leq n}Q_r(Y)X^r$ for some polynomials $Q_r$. For any degree $0\leq r\leq n$ and any of the above points $y_k$, let $c_{k,r}$ be the coefficient of $X^r$ of the polynomial $P(X,y_k)$. By Lagrange interpolation, one has 
\begin{equation}
Q_r(Y)=\sum_{k=0}^{n}c_{k,r}\prod_{\substack{s=0\\s\neq k}}^{n}\frac{Y-y_s}{y_k-y_s}.
\end{equation}
We write $T(Y)=\displaystyle{\prod_{\substack{s=0\\s\neq k}}^{n}(Y-y_s)=\sum_{t=0}^{n} a_t Y^t}$, and the relation between roots and coefficients (apply Lemma 19 page 310 of \cite{4} to the polynomial $T$) imply for all $0\leq t\leq n$ with nonzero $a_t$ 
\begin{equation}
\log\vert a_t\vert\leq \log \frac{L^n (2n)!}{n^n n!} + \frac{1+\log L}{L}n.
\end{equation}

Moreover 
\begin{equation}
\prod_{\substack{s=0\\s\neq k}}^n\vert y_k - y_s\vert=L^n \frac{(n-k)!k!}{n^n}.
\end{equation}

We get that the absolute value of the nonzero coefficients of $Q_r(Y)$ is bounded from above by 
\begin{equation}
\sum_{0\leq k\leq n} \vert c_{k,r}\vert \frac{L^n (2n)!}{n^n n!} e^{\frac{n}{L}(1+\log L)}\Big(\frac{(n-k)!k!}{n^n}L^n\Big)^{-1} \leq \max_{0\leq k\leq n} \vert c_{k,r}\vert e^{\frac{n}{L}(1+\log L)}\binom{2n}{n}2^n.
\end{equation}
As $h_\infty(P(X,y_k))\leq B$ by hypothesis, and $\binom{2n}{n}\leq 2^{2n}$, we are done.
\end{proof}

We add a small technical lemma.

\begin{lemma}\label{tech}
Let $a,b$ be real numbers satisfying the initial inequality $a\leq b+6\log(1+a)$. If $a\geq 47$, then $a\leq b+6\log(1+2b)$. \footnote{If $N\geq 46$ and $a\geq6N\log(1+N^2)$, one has $a\leq (1+\frac{1}{N-1})b$ and $a\leq b+6\log(1+\frac{N}{N-1}b)$. This leads to a slightly better upper bound in Corollary \ref{modul}, for bigger $m$ though.}
\end{lemma}

\begin{proof}
For any $a\geq 47$, one has $\log(1+a)\leq \frac{a}{12}$, hence if $a\leq b+6\log(1+a)$ we obtain $a\leq b+6\frac{a}{12}$, hence $a\leq 2b$. Inject in the initial inequality to obtain $a\leq b+6\log(1+2b)$. \footnote{Remark that depending on the value of $b$, one may want to bootstrap the initial inequality to $a\leq b+6\log(1+b+6\log(1+2b))$, but the gain is small here.}
\end{proof}

We are now ready to give the following general upper bound.

\begin{corollary}\label{modul}
Let $m$ be a positive integer. Then 
\begin{equation}
h_\infty(\Phi_m)\leq\psi(m)\Big(6\log m+\log\psi(m)+6\log(12\log m+2\log\psi(m)+25.2)+15.7  \Big).
\end{equation}
\end{corollary}

\begin{proof}

Let $j_0>1$ be a rational number. As $\Phi_m(X,j_0)$ is monic of degree $\psi(m)$ and with rational coefficients, the relation between roots and coefficients imply 
\begin{equation}\label{coef}
h_\infty(\Phi_m(X,j_0))\leq \log\binom{\psi(m)}{\lfloor\psi(m)/2\rfloor}+\psi(m)\max_{j}{h(j)},
\end{equation}
where $j$ ranges over all roots of $\Phi_m(X,j_0)$. Let $j_1$ be a root of $\Phi_m(X,j_0)$ of maximal height. Then Theorem \ref{dim1} implies 
\begin{equation}\label{roots}
 h(j_1)-h(j_0) \leq 10.68+6\log m +6\log(1+h(j_1)).
\end{equation}

If $h(j_1)\leq 47$, we get directly from (\ref{coef}) \begin{equation}\label{cas1}
h_\infty(\Phi_m(X,j_0))\leq \log\binom{\psi(m)}{\lfloor\psi(m)/2\rfloor}+\psi(m)47\leq 47.7\psi(m).
\end{equation}
If $h(j_1)\geq 47$, we use Lemma \ref{tech} with $a=h(j_1)$ and $b=h(j_0)+10.68+6\log m$ to write

\begin{equation}\label{roots2}
 h(j_1)\leq h(j_0)+10.68+6\log m +6\log(12\log m+2h(j_0)+22.36).
\end{equation}

Combine (\ref{coef}) and (\ref{roots2}) to obtain, using  $\binom{\psi(m)}{\lfloor\psi(m)/2\rfloor}\leq 2^{\psi(m)}$
\begin{equation}
\frac{h_\infty(\Phi_m(X,j_0))}{\psi(m)}\leq 6\log m+6\log(12\log m+2h(j_0)+22.36)+ h(j_0) + 10.68+\log2.
\end{equation}
Let us denote 
\begin{equation}
B(m,j_0)=\psi(m)\Big(6\log m+6\log(12\log m+2h(j_0)+22.36)+ h(j_0) + 11.38\Big).
\end{equation}
For any $0\leq k\leq n$, one has\footnote{This is where we lose an extra $\log m$. Other attempts for a set of interpolation points with smaller height give a better control here, but lose the more in (\ref{coef}).} $h(j_0(1+\frac{k}{n}))\leq h(j_0)+h(1+\frac{k}{n})\leq h(j_0)+\log(2n)$. Hence for any $0\leq k\leq n$, one gets

$\begin{array}{ll}
B(m,j_0(1+\frac{k}{n}))\leq & \psi(m)\Big(6\log m+6\log(12\log m+2h(j_0)+2\log(2n)+22.36)\\
& + h(j_0)+\log(2n) + 11.38\Big).\\
\end{array}$
\\

Let $B_2(m,j_0,n)$ stand for this last upper bound on $B(m,j_0(1+\frac{k}{n}))$, for any $0\leq k\leq n$. So in all cases, one can apply Lemma \ref{lem20} with $L=j_0$, the degree $n=\psi(m)$ and $B=\max\{B_2(m,j_0,\psi(m)), 47.7\psi(m)\}$ to obtain
\\

$\begin{array}{ll}
h_\infty(\Phi_m)\leq &  \max\{B_2(m,j_0,\psi(m)), 47.7\psi(m)\}+\psi(m)(\frac{1+ \log j_0}{j_0}+3\log2),
\\
& 
\end{array}$

and by choosing $j_0=2$ 
\\

$\begin{array}{ll}
h_\infty(\Phi_m)\leq & \max\Big\{\psi(m)\Big(6\log m+\log(2\psi(m))
\\

& +6\log(12\log m+2\log(2\psi(m))+23.8)+ 15\Big), 50.7\psi(m)\Big\},\\
\end{array}$

hence 
\\
\begin{equation}
\frac{h_\infty(\Phi_m)}{ \psi(m)}\leq\max\Big\{6\log m+\log\psi(m)+6\log(12\log m+2\log\psi(m)+25.2)+15.7, 50.7\Big\}.
\end{equation}

Denote $M(m)=\psi(m)\Big(6\log m+\log\psi(m)+6\log(12\log m+2\log\psi(m)+25.2)+15.7  \Big)$.

For $m\geq 6$, one has $\max\{M(m), 50.7\psi(m)\}=M(m)$ by direct estimate. For the first values of $m$ one has\\

$\begin{array}{l}
h_{\infty}(\Phi_1)=\log 1=0 \leq M(1),
\\
h_{\infty}(\Phi_2)=\log 157464000000000\leq M(2),
\\
h_{\infty}(\Phi_3)= \log 1855425871872000000000 \leq M(3),
\\
h_{\infty}(\Phi_4)= \log 280949374722195372109640625000000000000 \leq M(4),
\\
h_{\infty}(\Phi_5)= \log 141359947154721358697753474691071362751004672000 \leq M(5),
\\
\end{array}$

so we conclude that for any $m$ positive integer, the inequality $h_{\infty}(\Phi_m)\leq M(m)$ holds.
\end{proof}

\section{V\'elu's formulas}\label{section velu}

Let $E$ be an elliptic curve over $\overline{\mathbb{Q}}$ and let $G$ be a finite subgroup of $E(\overline{\mathbb{Q}})$. There exists an isogeny $\phi$ between $E$ and $E/G$, and an explicit equation for $E/G$ is given by V\'elu's work \cite{19} in the following way.

Choose a Weierstrass model of $E$: 
\begin{equation}
y^2+a_1xy+a_3y=x^3+a_2x^2+a_4x+a_6,
\end{equation}
and let $b_2=a_1^2+4a_2$, $b_4=a_1a_3+2a_4$, $b_6=a_3^2+4a_6$.
Let $F_2$ denote the set of points of order $2$ in $G$. One can find a subset $R$ in $G$ such that $G=F_2\cup R\cup (-R) \cup \{0\}$ as a disjoint union. Denote $S=R\cup F_2$. For $Q=(x_Q,y_Q)\in{S}$, we denote
 \\
 
 $\begin{array}{l}
 g_Q^x=3x_Q^2+2a_2x_Q+a_4-a_1y_Q,\\
 \\
 g_Q^y=-2y_Q-a_1x_Q-a_3,\\
 \\
 u_Q=4x_Q^3+b_2x_Q^2+2b_4x_Q+b_6,\\
 \\
 \end{array}
 $
 
  $\begin{array}{l}
 \displaystyle{t_Q=\left\{\begin{array} {l}
 x_Q^2+2a_2x_Q+a_4-a_1y_Q \quad if\quad Q\in F_2,\\
 6x_Q^2+b_2x_Q+b_4 \quad if\quad Q\in R,
 \end{array}\right.}\\
 \\
 \displaystyle{U_Q=\frac{t_Q}{x-x_Q}+\frac{u_Q}{(x-x_Q)^2},}\\
 \\
 \displaystyle{V_Q=\frac{2y+a_1x+a_3}{(x-x_Q)^3}+t_Q\frac{a_1(x-x_Q)+y-y_Q}{(x-x_Q)^2}+\frac{a_1u_Q-g_Q^xg_Q^y}{(x-x_Q)^2},}\\
 \\
 \displaystyle{t=\sum_{Q\in S}t_Q, \quad \mathrm{and}\quad\quad w=\sum_{Q\in S}(u_Q+x_Qt_Q).}\\
 \\
 \end{array}
 $
 
 The map $\Phi:E\to E/G$ given by $(x,y)\to(X,Y)=(x+\sum_{Q\in S}U_Q, y-\sum_{Q\in{S}} V_Q)$ is an isogeny of degree $\#G$ and one equation of $E/G$ is given by 
 \begin{equation}
 Y^2+a_1XY+a_3Y=X^3+a_2X^2+(a_4-5t)X+a_6-b_2t-7w.
 \end{equation}
 
In this situation, inequality (\ref{6}) of Theorem \ref{dim1} says the following.

\begin{corollary}
Let $E$ be an elliptic curve over $\overline{\mathbb{Q}}$ and let $G$ be a finite subgroup of $E(\overline{\mathbb{Q}})$. Let $j_E$ and $j_{E/G}$ denote their respective $j$-invariants, then
\begin{equation}\label{velu}
\vert h(j_E)-h(j_{E/G})\vert\leq 9.204+12\log\#G,
\end{equation}
\end{corollary}

In other words, when measuring the difference in size between equations for $E$ and for $E/G$, the number of elements of $G$ matters more than the size of the coordinates of the points of $G$. This is not obvious from V\'elu's construction.

\section*{Acknowledgments}

The author is supported by the DNRF Niels Bohr Professorship of Lars Hesselholt, ANR-14-CE25-0015 Gardio and ANR-17-CE40-0012 Flair. Merci \`a Pascal Autissier et Ga\"el R\'emond pour leurs remarques utiles. Many thanks to the referee for constructive feedback.


\begin{thebibliography}{0}

\bibitem{1} P. Autissier, Hauteur des correspondances de Hecke,
{\it Bull. Soc. Math. France} {\bf 131.3} (2003) 421--433.

\bibitem{2} B.C. Berndt and H.H. Chan, Ramanujan and the modular $j$-invariant, {\it Canad. Math. Bull.} {\bf 42.4} (1999) 427--440.

\bibitem{3} C. Birkenhake and H. Lange, Complex abelian varieties, 
  {\it Grundlehren der Mathematischen Wissenschaften } {\bf 302} (2004).

\bibitem{4} R. Br\"oker and A.V. Sutherland, An explicit height bound for the classical modular polynomial, {\it Ramanujan. J. } {\bf 22.3} (2010) 293--313.

\bibitem{5} J. Bruinier, K. Ono and A.V. Sutherland, Class polynomials for nonholomorphic modular functions, {\it Journal of Number Theory } {\bf 161} (2016) 204--229.

\bibitem{6} P. Cohen, On the coefficients of the transformation polynomials for the elliptic modular function, 
{\it Math. Proc. of the Cambridge Philo. Soc.} {\bf95} (1984) 389--402.

\bibitem{7} G. Cornell and J.H. Silverman (editors), Arithmetic Geometry, {\it Springer-Verlag} {\bf } (1986).

\bibitem{8} R. de Jong, On the Arakelov theory of elliptic curves, 
{\it l'Enseignement Math\'ematique} {\bf51} (2005) 179--201.

\bibitem{9} A. Enge and A.V. Sutherland, 
Class invariants by the CRT method,
{\it Ninth Algorithmic Number Theory Symposium ANTS-IX, Eds: Guillaume Hanrot and Fran\c{c}ois Morain and Emmanuel Thom\'e, Jul 2010, Nancy, France, Springer-Verlag, Lecture Notes in Computer Science} {\bf 6197} (2010) 142--156.

\bibitem{10} A. Faisant and G. Philibert, 
\textit{Mesure d'approximation pour la fonction modulaire $j$}.
{\it Publication de l'Universit\'e Pierre et Marie
Curie} {\bf 66.2} (1984) 54 pages.

\bibitem{11} A. Faisant and G. Philibert, 
Quelques r\'esultats de transcendance li\'es \`a l'invariant
modulaire $j$,
{\it J. Number Theory} {\bf 25.2} (1987) 184--200.

\bibitem{12} G. Faltings, 
Endlichkeitss\"atze f\"ur abelsche {V}ariet\"aten \"uber {Z}ahlk\"orpern, 
{\it Invent. Math.} {\bf73} (1983) 349--366.

\bibitem{13} G. Faltings, 
Calculus on arithmetic surfaces,
{\it Ann. of Math. (2)} {\bf119.2} (1984) 387--424.

\bibitem{14} E. Gaudron and G. R\'emond, 
Th\'eor\`eme des p\'eriodes et degr\'es minimaux d'isog\'enies,
{\it Comment. Math. Helvet.} {\bf 89.2} (2014) 343--403.

\bibitem{15} D.W. Masser and G. W\"ustholz, 
Some effective estimates for elliptic curves,
{\it Arithmetic of complex manifolds (Erlangen, 1988) Lecture Notes in Math.} {\bf 1399} (1989) 103--109.

\bibitem{16} D.W. Masser and G. W\"ustholz,  
Estimating isogenies on elliptic curves,
{\it Invent. Math.} {\bf 100.1} (1990) 1--24.

\bibitem{17} M. Raynaud,  
Hauteurs et isog\'enies,
{\it Ast\'erisque} {\bf127} (1985) 199--234.

\bibitem{18} J.H. Silverman,
Advanced topics in the arithmetic of elliptic curves,
{\it Graduate Texts in Mathematics} {\bf151} (1994).

\bibitem{19} J. V\'elu,
Isog\'enies entre courbes elliptiques,
{\it C.R. Acad. Sci. Paris S\'er. I Math.} {\bf273} (1971) 238--241.

\end{thebibliography}
\end{document}